\documentclass[reqno,12pt]{amsart}
\usepackage{amsmath, amssymb}
\usepackage{amsthm, amsbsy}

\PassOptionsToPackage{hyphens}{url}
\usepackage{hyperref} 
\hypersetup{colorlinks=true, linkcolor=blue, citecolor=magenta}

\theoremstyle{plain}
\newtheorem{theorem}{\indent\bf Theorem}[section]
\newtheorem{lemma}[theorem]{\indent\bf Lemma}
\newtheorem{corollary}[theorem]{\indent\bf Corollary}
\newtheorem{proposition}[theorem]{\indent\bf Proposition}

\newtheorem{orphtheorem}{\indent\bf Theorem}
\newtheorem{orphlemma}[orphtheorem]{\indent\bf Lemma}

\theoremstyle{definition}
\newtheorem{definition}[theorem]{\indent\bf Definition}
\newtheorem{remark}[theorem]{\indent\bf Remark}

\newtheorem{notation}[theorem]{\indent\bf Notation}

\numberwithin{equation}{section}

\DeclareMathOperator{\lcm}{\mathrm{lcm}}
\DeclareMathOperator{\ot}{\mathrm{o}}

\newcommand{\bb}[1]{\mathbb{#1}}

\begin{document}
	\title[The order of the product of two elements]{The order of the product of two elements in finite nilpotent groups}
	\author[C.M. Bonciocat]{Ciprian Mircea Bonciocat}
	\address{Address: University of California, Los Angeles, CA 90095}
	
	\email{cmbonciocat <at> ucla <dot> edu}

	
	\begin{abstract}
		An old problem in group theory is that of describing how the order of an element behaves under multiplication. To generalize some classical bounds concerning the order $\mathrm o(ab)$ of two elements $a, b$ in a finite abelian group to the non-commutative case, we replace $\mathrm o(ab)$ with a notion of mutual order $\mathrm o(a, b)$, defined as the least positive integer $n$ such that $a^nb^n = 1$. Motivated by this, we then compare $\mathrm o(ab)$ and $\mathrm o(a, b)$ in finite nilpotent groups, and show that in a group of class $\gamma$, the ratio $\mathrm o(ab)/\mathrm o(a, b)$ lies in some fixed finite set $S(\gamma) \subset \mathbb Q$, whose elements do not involve prime factors exceeding $\gamma$. In particular, we generalize a result of P. Hall, which asserts that $\mathrm o(ab) = \mathrm o(a, b)$ in $p$-groups with $p > \gamma$. We end with a more detailed analysis for groups of class 2, which allows one to give a more explicit description of $\mathrm o(ab)/\mathrm o(a, b)$.
	\end{abstract}

	\maketitle
	\setcounter{tocdepth}{2} 
	\let\oldtocsubsection=\tocsubsection
	\renewcommand{\tocsubsection}[2]{\hspace{1em}\oldtocsubsection{#1}{#2}}
	{\hypersetup{linkcolor=black}\tableofcontents} 
	\newpage
	
\section{Introduction} \label{SIntro}
	
	A challenging old problem in group theory is, given two elements $a, b$ in a group $G$, of orders $m$ and $n$ respectively, to find information on the order of the product $ab$. Understanding even the easier problem of when $ab$ has finite order would have great implications in group theory, for instance in the study of finitely-generated groups in which the generators have finite order. An example of a difficult problem related to this is Burnside's problem, which asks whether a finitely-generated periodic group is necessarily finite. A negative answer to this has been provided in 1964 by Golod and Shafarevich \cite{Golod}, \cite{GS}, although many variants of this question still remain unsolved to this day. For more information on this subject, we mention a few standard references: Kostrikin \cite{Kostrikin},  Novikov and Adian \cite{Adian}, Ivanov and Ol'shanskii \cite{Ivanov}, \cite{IvOls}, \cite{Ols}, Zelmanov \cite{Zelmanov1}, \cite{Zelmanov2} and Lys\"enok \cite{Lysenok}.
	
	Even in the simple case of abelian groups, no effective formula for the order of $ab$ seems to be available in the literature. We present here two well-known exercises from classical texts in abstract algebra, which address this problem in very special cases:
	
	\begin{orphlemma}\label{orphlema1}
		Let $a$ and $b$ be two elements of a finite abelian group with
		orders $m$ and $n$, respectively. If $m$ and $n$ are co-prime, then $ab$ has
		order $mn$.
	\end{orphlemma}
	
	\begin{orphlemma}\label{orphlema2}
		Let $x$ be an element of finite order $n$ in an arbitrary group, and $k$ an arbitrary integer. Then $x^k$ has order $n/\gcd(n,k)$.
	\end{orphlemma}

	A moment's reflection should convince the reader that no formula for $\ot(ab)$ can be given solely in terms of $m$ and $n$, in the general case. Of course, some divisibilities involving $m, n$ and $\ot(ab)$ can still be obtained with little effort, such as
	\begin{equation}\label{weakbound}
		\frac{\lcm(m, n)}{\gcd(m, n)} \mid \ot(ab) \mid \lcm(m, n).
	\end{equation}
	
	For a detailed analysis of what can be said about $\ot(ab)$ when the value $e := |\langle a \rangle \cap \langle b \rangle|$ is also known, an excellent reference is D. Jungnickel's paper \cite{Jungnickel1}. We present here three theorems from this source (using slightly modified notation), which demonstrate the complicated arithmetic involved in this problem:
	
	\begin{orphtheorem} \label{JungA} \cite[Theorem 1]{Jungnickel1}
		Let $a$ and $b$ be two elements in a finite commutative group $G$ with orders $m$ and $n$, respectively. Denote the subgroups of $G$ generated by $a$ and $b$ by $A$ and $B$, respectively, and assume that $A\cap B$ has order $e$ (where $e$ divides $\gcd(m,n)$). Let $D$ be the largest divisor of $e$ that is coprime to $m/\gcd(m,n)$ and $n/\gcd(m,n)$. Then the order of the product $ab$ satisfies
		\begin{equation*}
			\frac{{\lcm(m,n)}}{D}\mid \ot(ab) \mid \frac{{\lcm(m,n)}}{\varepsilon },
		\end{equation*}
		where $\varepsilon =1$ if $D$ is odd and $\varepsilon =2$ otherwise.
	\end{orphtheorem}
	
	\begin{orphtheorem} \label{JungB} \cite[Theorem 2]{Jungnickel1} 
		Let $m, n$ and $e$ be arbitrary positive integers for which $e$ divides both $m$ and $n$. Then there exists a finite abelian group $G$ with cyclic subgroups $A$ and $B$ of orders $m$ and $n$, respectively, $A\cap B$ has order $e$, and there exist generators $a, a'$ of $A$ and $b, b'$ of $B$ that satisfy
		\begin{equation*}
			\ot(ab)=\frac{{\lcm(m,n)}}{D}\quad {\rm and}\quad 
			\ot(a'b')=\frac{{\lcm(m,n)}}{\varepsilon },
		\end{equation*} 
		where $D$ and $\varepsilon $ are defined in Theorem \ref{JungA}.
	\end{orphtheorem}
	
	\begin{orphtheorem} \label{JungC} \cite[Theorem 3]{Jungnickel1}
		Let $m$ and $n$ be arbitrary positive integers, and let $k$ be any positive integer satisfying
		\begin{equation*}
			\frac{{\lcm(m,n)}}{f}\mid k\mid {\lcm(m,n)},
		\end{equation*}
		where $f$ is the largest divisor of $\gcd(m,n)$ which is co-prime to both $m/\gcd(m,n)$ and $n/\gcd(m,n)$.
		Then there exists a finite abelian group $G$ and elements $a, b$ of $G$ with orders $m$ and $n$, respectively, such that $\ot(ab)=k$.
	\end{orphtheorem}
	
	We also mention that an algorithm for determining the integer $D$ defined in Theorem \ref{JungA} can be found in L\"uneburg \cite[Ch. IV]{Luneburg}. 
	
	Despite its reduced scope, Lemma \ref{orphlema1} contributes to the proof of numerous foundational results in group theory and number theory. For instance, Lemma \ref{orphlema1} can be used to show that a finite group is cyclic if and only if its exponent and its order are equal. In turn, one may use this result to show that a finite subgroup of the multiplicative group of a field must be cyclic (see for instance Jacobson \cite[Theorems 1.4 and 2.18]{Jacobson} or van der Waerden \cite[paragraphs 42 and 43]{Waerden}). Lemma \ref{orphlema1} is also the basis of the famous algorithm due to Gauss that allows one to determine primitive elements in a finite field (that is generators for the cyclic multiplicative group) and then primitive polynomials (see for instance Jungnickel \cite[paragraph 2.5]{Jungnickel3}).
	
	In the general case of arbitrary groups, our hopes to find information on the order of $ab$ in terms of $m,n$ and possibly other information on the structure of $G$ are lowered by the following elegant result in Milne \cite[Theorem 1.64]{Milne}, showing that in this respect, essentially anything could happen.
	
	\begin{orphtheorem}\label{sepoateorice}
		For any integers $m,n,r>1$, there exists a finite group $G$ with elements $a$ and $b$ such that $a$ has order $m$, $b$ has order $n$, and $ab$ has order $r$.
	\end{orphtheorem}
	
	Its proof gives an explicit construction in quotients of $\mathrm{SL}(2, \bb F_q)$. This apparently random behavior can also be seen in a conjecture of Stefan Kohl \cite[problem 18.49]{Mazurov} recently proved independently by J. K\"onig \cite{Konig} and J. Pan \cite{Pan} to the effect that for any $x,y,z\in\mathbb{N}$ with $1<x,y,z\le n-2$ there exist $a, b\in S_{n}$ such that $a$ has order $x$, $b$ has order $y$, and $ab$ has order $z$. It should be noted that all proofs of Theorem \ref{sepoateorice} that the author has seen involve non-solvable groups.
	
	In the case of nilpotent groups, a remarkable fact still holds true: the product of two elements of finite order has finite order. This is in fact a corollary of a much more general result due to A. I. Mal'cev: 
	
	\begin{orphtheorem}\emph{(A. I. Mal'cev, \cite[2.23]{Clement})} 
		Let $G$ be a finitely-generated nilpotent group, containing a subgroup $H \le G$. If $G$ admits a finite generating set $X$ with the property that each element of $X$ has a power inside $H$, then any element of $G$ has a power inside $H$. If this is the case, then $H$ must have finite index in $G$.
	\end{orphtheorem}
	
	Indeed, by plugging $H = 1$ into the above, we learn that any nilpotent group $G$ which admits a finite generating set of torsion elements is finite. So if $a, b$ are elements of finite order in a nilpotent group $G$, then $ab \in \langle a, b \rangle$ must also have finite order. Note also that the case $H = 1$ of this theorem gives an affirmative answer to Burnside's problem in the special case of nilpotent groups. This was in fact previously noticed by R. Baer, in his paper \cite{Baer}. Another very short proof of the fact that the torsion part of a nilpotent group is a subgroup can be found in a post from Math StackExchange \cite{Magidin}.
	
	One may now naturally ask whether a bound such as (\ref{weakbound}) exists in the case of nilpotent groups. From the sources cited above, it seems that the best result one can obtain with the same methods is
	\begin{equation*}
		\ot(ab) \mid \lcm(m, n)^{\gamma},
	\end{equation*}
	where $\gamma$ is the nilpotency class of $G$. So if the class $\gamma$ is fixed, this gives a polynomial bound on $\ot(ab)$, in terms of $\lcm(m, n)$. In this paper, we show that this can actually be sharpened to a linear bound! In particular, we will extend a result of P. Hall \cite[4.28, 4.13]{Hall1} which, when modified to fit our paradigm, implies that $\ot(ab) \mid \lcm(m, n)$ for regular $p$-groups.
	
	Before stating some of our main results, we will fix some notations. For two elements $a$, $b$ of a group, we will always denote by $[a, b]$ their commutator $a^{-1}b^{-1}ab$, following the convention used in Gorenstein \cite{Gorenstein} and P. Hall \cite{Hall1}. In particular, we will often use the following trivial manipulations in our proofs:
	\begin{equation*}
		ba = ab \cdot [b, a], \qquad b^{-1}ab = a \cdot [a, b].
	\end{equation*}
	As usual, for a group $G$ we will denote by $Z(G)$ its center, by $G' = [G, G]$ its commutator subgroup, and by $G^{\rm ab} = G/[G, G]$ its abelianization. Also, $C_{G}(g)$ will denote the centralizer of an element $g$ of $G$. More elaborate notations and coventions will appear in section \ref{SNilp}, where we make use of Hall's complex commutators. Also, to prevent any potential errors, we assume all groups involved to be finite, unless stated otherwise.
	
	The theorem \ref{JungA} of D. Jungnickel on $\ot(ab)$ in finite abelian groups will appear as a consequence of a more general result, which holds in arbitrary finite groups. To state it, we recall that in \cite[4.28]{Hall1}, P. Hall proved that in regular $p$-groups the order of the product of two elements $a,b$ coincides with the least positive integer $N$ such that $a^{N}b^{N}=1$. This suggests the use of the following definition.
		
	\begin{definition}\label{defo(a,b)}
		For any two elements of finite order $a, b$ in an arbitrary group $G$, we denote by $\ot(a, b)$ the least positive integer $N$ satisfying $a^Nb^N = 1$, and call it {\it the mutual order of $a,b$}. Since the set $\{N \in \mathbb Z : a^Nb^N = 1\}$ is a nontrivial subgroup of $\mathbb Z$, the value $\ot(a, b)$ is a generator of this subgroup. In other words, the following equivalence holds:
		\begin{equation*}
			a^Nb^N = 1 \iff \ot(a, b) \mid N.
		\end{equation*}
	\end{definition}
	
	It is quite obvious from the definition that $\ot(a, b) = \ot(b, a)$ and $\ot(x, 1) = \ot(x)$. In the case that $a$ and $b$ commute, we also $\ot(a, b) = \ot(ab)$ have. Thus, this new notion generalizes the usual notion of order of an element. Another interesting observation is that $\ot(a^{-1}, b^{-1}ab)$ is the least power of $a$ that commutes with $b$.

	Our main goal in this paper is to explore the relationship between $\ot(ab)$ and $\ot(a, b)$ in the way more difficult case that $a$ and $b$ do not commute. The results that we will prove in Section \ref{SNilp} for nilpotent groups of arbitrary class rely on a remarkable formula of P. Hall (\cite[3.1, (3.21)]{Hall1}) expressing the powers of a product of two elements in terms of the powers of these elements and of their higher commutators. Our main result in this respect essentially says that in finite nilpotent groups, the order of the product of two elements is the same as their mutual order, modulo a factor that can be reasonably controlled:
	
	\begin{theorem}\label{constantacontrolabila}
		Let $G$ be a finite nilpotent group of class $\gamma $. There exists a finite set of positive rational numbers $S = S(\gamma)$, depending solely on the nilpotency class $\gamma$, and whose elements in reduced form contain no primes exceeding $\gamma$ in their numerators and denominators, such that for all $a, b$ in $G$, there exists $s \in S$ depending on $a$ and $b$ such that
		\begin{equation*}
			\ot(ab) = \ot(a, b) \cdot s.
		\end{equation*}
	\end{theorem}
	
	In particular, we obtain as a corollary the following famous result of P. Hall:
	
	\begin{corollary}\label{corolarcugamma}\cite[4.28, 4.13]{Hall1}
		If $G$ is a finite nilpotent group of class smaller than any prime dividing the order of $G$, then
		\begin{equation*}
			\ot(ab) = \ot(a, b)
		\end{equation*}
		for any two elements $a, b \in G$. 
	\end{corollary}

	Combined with Jungnickel's Theorem \ref{JungA}, we obtain the promised bounds for $\ot(ab)$, linear in $\lcm(m, n)$. For nilpotent groups of class 2 we will be able to prove in Section \ref{SClass2} more effective results, since in this case $\ot(ab)/\ot(a,b)\in\{\frac{1}{2},1,2\}$. Our main result in this respect is:
	
	\begin{theorem}\label{IntroducereClass2}
		Let $a, b$ be elements of finite order in a nilpotent group of class 2, and let $r$ be the order of $c = [b, a] := b^{-1}a^{-1}ba$. Then $r$ divides both $\ot(ab)$ and $\ot(a, b)$, and one has the formula
		\begin{equation}
			\ot(ab) = \ot(a, b) \cdot \frac{\ot(a^rb^rc^{r \choose 2})}{\ot(a^rb^r)},
		\end{equation}
		where the factor $\frac{\ot(a^rb^rc^{r \choose 2})}{\ot(a^rb^r)}$ lies in the set $\{\frac 12, 1, 2\}$.
	\end{theorem}

	A detailed analysis of the three possible cases above will end Section \ref{SClass2}, and will give an insight on how difficult it might be to search for exact formulas for $\ot(ab)$ in finite groups of higher nilpotency class. Such formulas might be in principle possible to obtain, but at the cost of a way more difficult analysis, requiring an increasing number of parameters.
	
	\section{Properties of the mutual order $\mathrm{o}(a,b)$ in finite groups}\label{Formulaot(a,b)} 

	This section is concerned with proving some elementary facts about the mutual order $\ot(a, b)$, as well as generalizing Jungnickel's Theorem \ref{JungA} to arbitrary finite groups. We advise the experienced reader to skim through this section, as most proofs contained here are quite straight-forward. We start with a simple fact, which generalizes Lemma \ref{orphlema2} to mutual orders:

	\begin{proposition}\label{lema2}
		Let $a, b$ be elements of finite orders in an arbitrary group. Then
		\begin{equation*}
			\ot(a^n, b^n) = \frac{\ot(a, b)}{\gcd(\ot(a, b), n)},
		\end{equation*}
		for all integers $n$.
	\end{proposition}
	
	\begin{proof}
		First of all, note that $n \cdot \frac{\ot(a, b)}{\gcd(\ot(a, b), n)}$ is a multiple of $\ot(a, b)$, so
		\begin{equation*}
			1 = a^{n \cdot \frac{\ot(a, b)}{\gcd(\ot(a, b), n)}}b^{n \cdot \frac{\ot(a, b)}{\gcd(\ot(a, b), n)}} = (a^n)^{\frac{\ot(a, b)}{\gcd(\ot(a, b), n)}}(b^n)^{\frac{\ot(a, b)}{\gcd(\ot(a, b), n)}}.
		\end{equation*}
		This means that $\ot(a^n, b^n) \mid \frac{\ot(a, b)}{\gcd(\ot(a, b), n)}$. For the converse divisibility, we observe that the equality $1 = a^{n \cdot \ot(a^n, b^n)}b^{n \cdot \ot(a^n, b^n)}$ implies $\ot(a, b) \mid n \cdot \ot(a^n, b^n)$, from which we further deduce that
		\begin{equation*}
			\frac{\ot(a, b)}{\gcd(\ot(a, b), n)} \mathrel{\Big|} \frac{n}{\gcd(\ot(a, b), n)} \cdot \ot(a^n, b^n).
		\end{equation*}
		The conclusion comes now from the fact that $\frac{\ot(a, b)}{\gcd(\ot(a, b), n)}$ and $\frac{n}{\gcd(\ot(a, b), n)}$ are coprime.
	\end{proof}
	
	We also state separately the particular case that $s$ is a divisor of $\ot(a,b)$. This will be useful in many cases where we know a divisor $d$ of $\ot(a, b)$, and we try to understand $\ot(a, b)$ in terms of the potentially simpler quantity $\ot(a^d, b^d)$:
	
	\begin{corollary}\label{reduce}
		Let $a, b$ be elements of finite orders in an arbitrary group, and let $s$ be a divisor of $\ot(a, b)$. Then $\ot(a, b) = s \cdot \ot(a^s, b^s)$.
	\end{corollary}

	We are now ready to state and prove the generalization of Theorem \ref{JungA} to arbitrary finite groups. The only difference is that $\ot(ab)$ is replaced with $\ot(a, b)$:
	
	\begin{theorem}\label{JungAGen}
		Let $a$ and $b$ be two elements in a finite group $G$ with orders $m$ and $n$, respectively. Denote the subgroups of $G$ generated by $a$ and $b$ by $A$ and $B$, respectively, and assume that $A \cap B$ has order $e$ (where $e$ divides $\gcd(m,n)$). Let $D$ be the largest divisor of $e$ that is coprime to $m/\gcd(m,n)$ and $n/\gcd(m,n)$. Then the mutual order $\ot(a, b)$ satisfies
		\begin{equation*}
			\frac{{\lcm(m,n)}}{D}\mid \ot(a, b) \mid \frac{{\lcm(m,n)}}{\varepsilon },
		\end{equation*}
		where $\varepsilon = 1$ if $D$ is odd and $\varepsilon = 2$ otherwise.
	\end{theorem}

	\begin{proof}
		We start by deducing a slightly more explicit expression for $\ot(a, b)$. First of all, let us point out that both $\frac me$ and $\frac ne$ divide $\ot(a, b)$. Indeed, note that we have a chain of implications
		\begin{equation*}
			a^Nb^N = 1 \Longrightarrow a^N = b^{-N} \in \langle a \rangle \cap \langle b \rangle \Longrightarrow a^{eN} = b^{-eN} = 1.
		\end{equation*}
		So by the the definition of the order, we further obtain
		\begin{equation*}
			m = \ot(a) \mid eN \Longrightarrow \frac me \mid N, \text{\; and\; } n = \ot(b) \mid eN \Longrightarrow \frac ne \mid N.
		\end{equation*}
		Also, let us note that $a^{\frac me}$ and $b^{\frac ne}$ both generate the intersection subgroup $\langle a \rangle \cap \langle b \rangle$. We now apply Corollary \ref{reduce}, with $s = \frac{\lcm(m, n)}{e}$:
		\begin{equation*}
			\ot(a, b) = \frac {\lcm(m, n)}e \cdot \ot\!\Big(a^\frac {\lcm(m, n)}e, b^{\frac {\lcm(m, n)}e}\Big) = \frac {\lcm(m, n)}e \cdot \ot\!\Big(a^\frac {\lcm(m, n)}eb^{\frac {\lcm(m, n)}e}\Big).
		\end{equation*}
		The last step comes from the fact that $a^{\frac {\lcm(m, n)}e}, b^{\frac {\lcm(m, n)}e}$ are both elements in the abelian group $\langle a \rangle \cap \langle b \rangle$. Now if $g$ denotes any generator of $\langle a \rangle \cap \langle b \rangle$, the other two generators $a^{\frac me}, b^{\frac ne}$ can be written as $g^u, g^v$ respectively, with $u, v$ coprime to $e = |\langle a \rangle \cap \langle b \rangle|$.\footnote{Note that $g$ may be taken to be one of $a^{\frac me}$ and $b^{\frac ne}$, in which case one of $u$ and $v$ becomes 1. However, wishing to keep everything symmetric, we employ this more general notation.} With this notation, $\ot(a, b)$ can be further expressed as
		\begin{equation*}
			\ot(a, b) = \frac{\lcm(m, n)}e\cdot \ot\!\Big(g^{\frac{vm + un}{\gcd(m, n)}}\Big) \stackrel{\ref{lema2}}= \frac{\lcm(m, n)}{\gcd(e, \frac{vm + un}{\gcd(m, n)})}.
		\end{equation*}
		For brevity, we denote $m' := \frac m{\gcd(m, n)}$ and $n' := \frac n{\gcd(m, n)}$. It now remains to show the following divisibility:
		\begin{equation*}
			\varepsilon \mid \gcd(e, vm' + un') \mid D.
		\end{equation*}
		We start with the left side, since it is a bit easier to see. Note that the only non-trivial content of this divisibility is when $D$ is even, and $\varepsilon = 2$. In this case, $e$ is also even, so we must see that $um' + vn'$ is even as well. The numbers $u$ and $v$ are odd because they are both coprime to $e$, while $m', n'$ are odd because they are both coprime to $D$. So indeed $vm' + un'$ is even.
		
		Now for the second divisibility, let $d$ be a common divisor of $e$ and $vm' + un'$, and let us see why $d$ must in fact divide $D$. If by absurd $d$ has any prime factor $p$ in common with $m'$ it would follow that
		\begin{equation*}
			 p \mid (vm' + un') - vm' = un'.
		\end{equation*}
		But $m', n'$ are coprime by definition, whereas $u$ and $e$ are coprime since $\langle g^u\rangle = \langle g \rangle$. So this is indeed a contradiction, as $p$ can divide neither of $u$ and $n'$. So $d$ is a divisor of $e$, coprime with both of $m'$ and $n'$, i.e. $d \mid D$ as wished. This concludes our proof.
	\end{proof}

	\begin{remark}
		The $u$ and $v$ in the proof above obviously depend on the chosen generator $g$ of $\langle a \rangle \cap \langle b \rangle$. However, there is a way to replace the choice of the pair $(u, v)$ with something else, that is both symmetric and canonical. Indeed, note that if we change $g$ to another generator $g'$ satisfying $g = (g')^k$, then the new pair consists of $u' \equiv uk \; (\text{mod } e)$ and $v' \equiv vk \; (\text{mod } e)$. So we can construct a set similar to the projective space, given by the orbits of the diagonal action of $(\bb Z/e\bb Z)^\times$ on $\bb Z/e\bb Z \times \bb Z/e\bb Z$:
		\begin{equation*}
			\bb P_e := \frac{\{(a, b) \in \bb Z/e\bb Z \times \bb Z/e\bb Z\}}{(ka, kb) \sim (a, b), \forall k \in (\bb Z/e\bb Z)^\times}.
		\end{equation*}
		It is now easy to see that the equivalence class of $(u, v)$ in $\bb P_e$ is well-defined, no matter what $g$ is. Also, note that for any point $\pi \in \bb P_e$ represented by a pair $(a, b)$, there is a well-defined evaluation map on $\bb Z \times \bb Z$, given by
		\begin{equation*}
			\pi(x, y) \longmapsto \gcd(e, bx + ay),
		\end{equation*}
		since units modulo $e$ do not affect the gcd. So if $\pi$ is the equivalence class of $(u, v)$ from before, we get a canonical, symmetric ``formula''
		\begin{equation*}
			\ot(a, b) = \frac{\lcm(m, n)}{\pi(m', n')}.
		\end{equation*}
		Of course, if one is satisfied with an asymmetric formula, then a canonical choice can be made simply by choosing the pair $(u, v)$ such that $u = 1$.
	\end{remark}

	We end this section with a few corollaries, most of which are concerned with when the value $\ot(a, b)$ is uniquely determined. It is easy to see (using also Theorem \ref{JungB} of Jungnickel) that this happens precisely when the equality $\varepsilon = D$ occurs, i.e. when $D \in \{1, 2\}$. Since we want to make no mention of $e$ in the statements, we wish to find all pairs $(m, n)$ such that all choices of $e \mid \gcd(m, n)$ lead to $D \in \{1, 2\}$. It is not hard to show that this happens if and only the $D$ associated to $e = \gcd(m, n)$ is in $\{1, 2\}$. Thus, we have
	
	\begin{corollary}
		Given positive integers $m, n$ such that the biggest divisor $D$ of $\gcd(m, n)$ coprime to each of $m'$ and $n'$ is 1 or 2, we have
		\begin{equation*}
			\ot(a, b) = \frac{\lcm(m, n)}{D},
		\end{equation*}
		for all possible choices of elements $a, b$ of orders $m, n$ respectively.
	\end{corollary}

	In particular, we get the following easy to remember corollaries, whose proofs may potentially be obtained through easier methods as well:
	
	\begin{corollary} \label{coro2}
		If $m, n$ are positive integers such that $v_p(m) \neq v_p(n)$ for all $p \mid \gcd(m, n)$, then 
		\begin{equation*}
			\ot(a, b) = \lcm(m, n),
		\end{equation*}
		for all possible choices of elements $a, b$ of orders $m, n$ respectively.
	\end{corollary}

	\begin{corollary}
		If $0 \le \alpha < \beta$ are integers, then
		\begin{equation*}
			\ot(a, b) = \ot(b) = p^\beta,
		\end{equation*}
		for all possible choices of elements $a, b$ of orders $p^\alpha, p^\beta$ respectively.
	\end{corollary}
	
	\section{Analysis of the ratio $\mathrm{o}(ab)/\mathrm{o}(a, b)$ in the nilpotent case}\label{SNilp}
	
	Given any two elements $a,b$ of an arbitrary group $G$ we may write $(ab)^n=a^nb^n\cdot d_n(a,b)$ with $d_n(a,b)$ an element in the derived subgroup $G'$. In order to study the relationship between $\ot(ab)$ and $\ot(a,b)$ it is therefore useful to find information on the elements $d_n$. It is easy to check that $d_n$ satisfies the recurrence relation $d_n(a,b)=([b^{n-1},a^{n-1}]\cdot d_{n-1}(b,a))^b$. Unrolling this recurrence relation is easily seen to give an expression of $d_n$ as a product of $n-1$ conjugates of $[x^i,y^i]^{\pm 1}$. It is in general desirable to find the shortest possible expression of $d_n$, or of an arbitrary element of $G'$ as a product of commutators, which is the so called commutator length problem. This, together with the stable version asking to describe the limit of the $\frac 1n$th of the commutator length of the $n$th power, are notoriously difficult problems with ramifications in low-dimensional manifolds, symplectic topology, dynamics, and in the theory of quasi-imorphisms and of bounded cohomology. For these topics we refer the interested reader to the fundamental work of Culler \cite{Culler}, Bavard \cite{Bavard} and Calegari \cite{Calegari}. In the case of nilpotent groups it is often useful to investigate $d_n$ by using the famous Hall polynomials and their properties. This will be our approach here, requiring the following notation.

	\begin{notation}\label{notatiaHall}
		Following the definitions in P. Hall \cite{Hall1}, we will present the notion of \emph{complex commutators} in the symbols $a$ and $b$. These are defined inductively as follows: 
		\begin{enumerate}
			\item The complex commutators of weight 1 are the symbols $a$ and $b$ themselves.
			\item Assuming the complex commutators of weights $1, \ldots, w - 1$ have already been defined, we define a complex commutator of weight $w$ to be any expression of the form $[S, T]$, where $S, T$ are complex commutators of lower weights $w_1, w_2$, satisfying $w_1 + w_2 = w$.
		\end{enumerate}
		The reason we use words such as ``symbol'' and ``expression'' is because we are viewing these complex commutators as formal operations in $a$ and $b$, rather than actual elements in $G$. That is, even if two complex commutators give the same value when applied to some concrete values $a, b$ in a group $G$, they may not be formally equal.\footnote{In the same way that polynomials in $\mathbb F_p[x]$ may be equal as functions, but not as formal polynomials.} In \cite[p. 43]{Hall1}, P. Hall does not explicitly state this in his definition of complex commutators, but then proceeds to say ``formally distinct complex commutators'' in any later result where this matters. Concrete examples of \emph{distinct} complex commutators include
		\begin{equation*}
			[a, b],\quad [a, a], \quad [b, b], \quad [[a, b], [b, a]].
		\end{equation*}
		The weight can be understood as the total number of symbols from the set $\{a, b\}$ that appear when writing out the commutators explicitly. For any complex commutator $c$, we will denote its weight by $w(c)$.
		
		Obviously, some of the complex commutators will always give the value 1 when evaluated at concrete values (for example $[a, a]$). We say that a complex commutator is degenerate if at some point in its construction a complex commutator of the form $[x, x]$ appears. For example, $[a, [[a, b], [a, b]]]$ is considered degenerate. An easy induction shows that indeed any degenerate complex commutator gives value 1 whenever evaluated at some concrete values $a, b$.\footnote{Note that there exist complex commutators that always take value 1 even if they are not degenerate: for instance $[[a, b], [b, a]]$ always gives value 1, as $[a, b] = [b, a]^{-1}$. While potentially better notions of degeneracy may be defined, the current definition is sufficient for our purposes.} 
		
		Next we need to introduce a total ordering relation on all complex commutators. Following the conventions from \cite{Hall1}, we order them in increasing order of their weights, and allow the ordering among commutators of equal weight to be arbitrary. That is, we may take
		\begin{equation} \label{ordine}
			c_0 = a, c_1 = b, c_2, c_3, \ldots, c_i, \ldots
		\end{equation}
		to be a sequence containing all complex commutators in $a, b$, such that $w(c_i) \le w(c_j)$ whenever $i \le j$. This is legal because for any given weight $w$, there exist only finitely many $c_i$ of weight at most $w$. For convenience we make the notation $w_i = w(c_i)$. Now we fix an ordering of the form (\ref{ordine}), which uniquely assigns an index to any complex commutator.
		
		In what follows, let $G$ be a finite nilpotent group of fixed class $\gamma$, generated by two elements $a, b$. Now that the ordering is fixed, we can finally view $c_i$ as actual elements of $G$, and not just as formal expressions. By the previous observation that there are only finitely many $c_i$ of a given weight $w$, there exists a greatest index $r$ such that $c_r$ has weight $\gamma$. Again since the ordering (\ref{ordine}) is fixed, this number $r$ depends only on $\gamma$. Because the nilpotency class is $\gamma$, the commutators $c_k$ with $k > r$ all vanish (\cite[2.53]{Hall1}), so effectively only $c_0, c_1, \ldots, c_r$ will be relevant.
	\end{notation}
	
	With this notation in mind, we have the following celebrated formula of P. Hall, expressing the powers of a product of two elements in terms of the powers of these elements and of their higher complex commutators.
	
	\begin{theorem}\cite[3.1, (3.21)]{Hall1}\label{hallform}
		For any integer $n$ one has the formula
		\begin{equation*}
			(ab)^n = a^nb^nc_2^{f_2(n)}c_3^{f_3(n)} \cdots c_r^{f_r(n)},
		\end{equation*}
		where $f_k$ $(2 \le k \le r)$ are polynomials that can be written as
		\begin{equation*}
			f_k(x) = \lambda_{k, 1}{x \choose 1} + \lambda_{k, 2}{x \choose 2} + \cdots + \lambda_{k, w_k}{x \choose w_k},
		\end{equation*}
		with integer constants $\lambda_{k, \ell}$ depending only on the subscripts $k$ and $\ell$.
	\end{theorem}

	We point out that each $f_k$ ($2 \le k \le r$) is a polynomial without free term, and with the least common denominator of its coefficients dividing $\gamma!$. As a result, whenever $\gamma ! \mid X$, we have the divisibility
	\begin{equation}\label{divizfact}
		\tfrac X{\gamma!} \mid f_k(X), \; \text{ for all } \; k \in \{2, \ldots, r\}.
	\end{equation} 
	We recall that $f_k$ are well-defined only after a given ordering of the complex commutators has been fixed. Of course, once this choice is made, $f_k$ is now a fixed polynomial, which does not depend on $\gamma$. It is also important to note that the degenerate commutators may be removed from the sequence $(c_i)_{i \ge 0}$, since they do not contribute at all to the formula. However, not wishing to alter the original statement of this theorem, we leave the degenerate commutators there as well. For the proof of Theorem \ref{constanteleBC} we will need the following technical lemma, which might be of independent interest and useful in other applications.
	\begin{lemma}\label{lemcenter}
		There exists a positive integer $A = A(\gamma)$ depending solely on the nilpotency class $\gamma$ of $G$, with prime factors at most $\gamma$, such that whenever
		\begin{equation}\label{powersofc}
			c_0^n, c_1^n, c_2^n, \ldots, c_r^n \in Z(G),
		\end{equation}
		for some integer $n$, we also have
		\begin{equation}\label{powersofc2}
			c_2^{n \cdot A} = c_3^{n \cdot A} = \cdots = c_r^{n \cdot A} = 1.
		\end{equation}
	\end{lemma}
	
	We stress the fact that in (\ref{powersofc}) the indexing starts from 0, while in (\ref{powersofc2}) it starts from 2. In other words, more effort must be put in to annihilate the powers of $c_2, \ldots, c_r$. This result will be crucial, as it will allow one to induct on the nilpotency class $\gamma$, reducing questions in $G$ to questions in $G/Z(G)$.
	
	\begin{proof}
		In what follows we will take $A = (\gamma !)^{r - 2}$, although potentially better uniform bounds could be found by a deeper analysis. We recall that $r$ is also fully dependent on $\gamma$, since it represents the greatest index of a commutator of weight $\gamma$. Therefore, our choice of $A$ is indeed a function only of $\gamma$, whose prime divisors do not exceed $\gamma$.
		
		We actually prove the slightly stronger result that 
		\begin{equation} \label{finer} 
			c_k^{n \cdot (\gamma!)^{r - k}} = 1 \quad \text{whenever} \quad 2 \le k \le r. 
		\end{equation}
		This will be shown by means of a downward induction, starting with the initial step $k = r$, and then going down until $k = 2$.
		
		First of all, note that for any $k$ in $\{2, \ldots, r\}$ the complex commutator $c_k$ can be expressed canonically as the commutator of some complex commutators $c_i, c_j$ with $0 \le i, j < k$, i.e. $c_ic_k = c_j^{-1}c_ic_j$. Raising this to a power $N$ divisible by $n$, we get
		\begin{equation} \label{cjck1} 
			(c_ic_k)^N = (c_j^{-1} c_i c_j)^N = c_j^{-1} c_i^N c_j = c_i^N, 
		\end{equation}
		since $c_i^N \in Z(G)$, according to (\ref{powersofc}). On the other hand, by expanding the power $(c_ic_k)^N$ as in Theorem \ref{hallform} we get
		\begin{equation} \label{cjck2} 
			(c_ic_k)^N = c_i^Nc_k^N d_2^{f_2(N)} d_3^{f_3(N)} \cdots d_r^{f_r(N)},
		\end{equation}
		where $d_\ell$ ($2 \le \ell \le r$) are complex commutators in the symbols $c_i, c_k$. In particular, they are also complex commutators in $a$ and $b$. If in the expression of $d_\ell$ the symbol $c_k$ appears at least once, then its weight as a commutator in $a, b$ exceeds that of $c_k$. Otherwise, only $c_i$'s are used, and the complex commutator is degenerate. Thus, all non-degenerate $d_\ell$ appear to the right of $c_k$ in the ordering (\ref{ordine}), when viewed as complex commutators in $a$ and $b$. Combining (\ref{cjck1}), (\ref{cjck2}) and canceling the $c_i^N$ yields
		\begin{equation} \label{reduct} 
			1 = c_k^N d_2^{f_2(N)} d_3^{f_3(N)} \cdots d_r^{f_r(N)}.
		\end{equation}
		We proceed now with the induction argument. If we are in the initial case $k = r$, then any non-degenerate commutator $d_\ell$ is at the right of $c_r$ in our ordering. Since $r$ is the largest index with $w_r \le \gamma$, we learn that in fact all $d_\ell$ are the identity. Plugging now $N = n$ in (\ref{reduct}) gives $c_r^n = 1$, which is precisely (\ref{finer}) for $k = r$, thus proving the initial step of the induction.
		
		Assuming now that the statement (\ref{finer}) has been proven for $r, r-1, \ldots, k + 1$, we wish to prove it for $k$. First of all, by applying (\ref{divizfact}) to the case $X = n \cdot (\gamma!)^{r - k}$, we get
		\begin{equation}\label{divizibilitatea}
			n \cdot (\gamma!)^{r - k - 1} \mid f_\ell(n \cdot (\gamma!)^{r - k}),
		\end{equation}
		for all $\ell$ between 2 and $r$. Since the power $n \cdot (\gamma!)^{r - k - 1}$ kills all commutators to the right of $c_k$ (by the inductive hypothesis), in particular it kills all non-degenerate $d_{\ell}$. Thus, in view of the divisibility (\ref{divizibilitatea}), one obtains that $d_\ell^{f_\ell(n \cdot (\gamma!)^{r - k})} = 1$ for all $\ell \in \{2, \ldots, r\}$. So all that remains in (\ref{reduct}) after plugging in $N = n \cdot (\gamma!)^{r - k}$ is $c_k^{n \cdot (\gamma!)^{r-k}} = 1$, i.e. (\ref{finer}) holds for $k$ as well, which concludes the inductive argument.
	\end{proof}

	We will now proceed with the following result that gives valuable information on the ratio $\ot(ab)/\ot(a,b)$ in finite nilpotent groups.
	
	\begin{theorem}\label{constanteleBC}
		Let $G$ be a finite nilpotent group of class $\gamma$. There exist two integer constants $B = B(\gamma)$ and $C = C(\gamma)$ depending solely on the nilpotency class $\gamma$, and having prime factors at most $\gamma$, such that
		\begin{equation}\label{celedouadivizibilitati}
			\ot(ab) \mid \ot(a, b) \cdot B\quad \text{and}\quad
			\ot(a, b) \mid \ot(ab) \cdot C
		\end{equation}
	for every elements $a, b$ in $G$.
	\end{theorem}

	\begin{proof}
		Note that we may assume that $G = \langle a, b \rangle$, without restricting the generality of the statement.
		
		(i) We first prove the existence of the constant $B$ with the desired properties. We will actually prove the stronger result that there exists an integer $B = B' \cdot \gamma!$ so that the power $\ot(a, b) \cdot B'$ kills all commutators $c_k$, for $k \ge 2$ (and so that $B' = B'(\gamma)$ has prime factors at most $\gamma$). To see that this indeed implies the desired conclusion, first note that each $f_k$ ($2 \le k \le r$) satisfies
		\begin{equation}\label{altadiv}
			\ot(a, b) \cdot B' \mid f_k(\ot(a, b) \cdot B),
		\end{equation}
		in view of (\ref{divizfact}). So
		\begin{equation*}
			(ab)^{\ot(a, b) \cdot B} = a^{\ot(a, b) \cdot B}b^{\ot(a, b) \cdot B} \cdot c_2^{f_2(\ot(a, b) \cdot B)} \cdots c_r^{f_r(\ot(a, b) \cdot B)} = a^{\ot(a, b) \cdot B}b^{\ot(a, b) \cdot B} = 1,
		\end{equation*}
		due to Theorem \ref{hallform}, relation (\ref{altadiv}), and the fact that $a^nb^n = 1$ for some integer $n$ if and only if $\ot(a, b) \mid n$. Consequently, $\ot(ab) \mid \ot(a, b) \cdot B$, as desired.
		
		We will prove this stronger result by induction on the nilpotency class $\gamma$. The initial step when $\gamma = 1$ refers to abelian groups, in which case all commutators naturally vanish, and one may take $B'(1)=1$. 
		So let us assume that the result is true for nilpotent groups of class at most $\gamma - 1$, and try to prove it for our group $G$ of class at most $\gamma$. In order to use the inductive hypothesis, we look at the images of the elements in the quotient group $G/Z(G)$, whose class is $\gamma -1$ (see \cite[Lemma 2.12]{Clement}, for instance). Indeed, if $\hat x$ represents the image of $x$ under the quotient map, we already know that
		\begin{equation*}
			\hat c_k^{\;\ot(\hat a, \hat b) \cdot B'(\gamma - 1)} = \hat 1\quad \text{ for all } k\ge 2.
		\end{equation*}
		Also, $\ot(\hat a, \hat b) \mid \ot(a, b)$ because $\hat a^{\ot(a, b)} \hat b^{\ot(a, b)} = \widehat{a^{\ot(a, b)} b^{\ot(a, b)}} = \hat 1$. So we can actually get rid of the hats on $a, b$ in the previous equation, to obtain
		\begin{equation*}
			\hat c_k^{\;\ot(a, b) \cdot B'(\gamma - 1)} = \hat 1.
		\end{equation*}
		This means that all $c_k$ ($2 \le k \le r$) raised to the power $\ot(a, b) \cdot B'(\gamma - 1)$ must enter the center $Z(G)$. In order to apply Lemma \ref{lemcenter} it remains to prove the same for $c_0 = a$ and $c_1 = b$. To do so, observe that $b^{\ot(a, b)} = a^{-\ot(a, b)}$ commutes with both $a$ and $b$, so both $a^{\ot(a, b) \cdot B'(\gamma - 1)}$ and $b^{\ot(a, b) \cdot B'(\gamma - 1)}$ are in the center, as $\langle a, b\rangle = G$. So now we may apply Lemma \ref{lemcenter} with $n = \ot(a, b) \cdot B'(\gamma - 1)$, to deduce that the power $\ot(a, b) \cdot B'(\gamma - 1) \cdot A(\gamma)$ kills all $c_k$ ($2 \le k \le r$). Thus, we may choose
		\begin{equation*}
			B'(\gamma) := B'(\gamma-1) \cdot A(\gamma),
		\end{equation*}
		and our induction is complete. Unwinding the recursive formula above, one obtains
		\begin{equation}\label{B(gama)}
			B(\gamma) = \gamma! \cdot B'(\gamma) = \gamma! \cdot A(2) \cdots A(\gamma),
		\end{equation}
		which obviously depends only on $\gamma$, and has no prime factors exceeding $\gamma$. This proves the first part of the theorem.
		
		(ii) We will now prove the existence of the integer constant $C$ with the desired properties. Much as in part (i), we will prove the stronger result that there exists $C = \gamma! \cdot C'$ such that the power $\ot(ab) \cdot C'$ kills all commutators $c_k$ with $2 \le k \le r$. To see that this is indeed a stronger result, first note that
		\begin{equation}\label{caramel}
			\ot(ab) \cdot C' \mid f_k(\ot(ab) \cdot C)
		\end{equation}
		for all $k$ in $\{2, \ldots, r\}$, which follows again from (\ref{divizfact}). Next, by Theorem \ref{hallform} we deduce that 
		\begin{equation*}
			1 = (ab)^{\ot(ab) \cdot C} = a^{\ot(ab) \cdot C}b^{\ot(ab) \cdot C} \cdot c_2^{f_2(\ot(ab) \cdot C)} \cdots c_r^{f_r(\ot(ab) \cdot C)} = a^{\ot(ab) \cdot C}b^{\ot(ab) \cdot C},
		\end{equation*}
		so $\ot(a, b) \mid \ot(ab) \cdot C$, which proves our claim.
		
		As before, we will prove this stronger result by induction on $\gamma$. The abelian case ($\gamma = 1$) is superfluous by taking $C'(1)=1$, since higher commutators are trivial by default. So let us assume that the result holds for nilpotent groups of class at most $\gamma - 1$, and prove it for groups of nilpotency class $\gamma$. If we denote by $\hat x$ the image of $x$ in the quotient $G/Z(G)$, whose class is $\gamma -1$, then the induction hypothesis tells us that
		\begin{equation*}
			\hat c_k^{\;\ot(\hat a \hat b) \cdot C'(\gamma - 1)} = \hat 1\; \text{ for all } \; k\in \{ 2,\dots ,r\} .
		\end{equation*}
		Since $(\hat a \hat b)^{\ot(ab)} = \widehat{(ab)^{\ot(ab)}} = \hat 1$, we obtain $\ot(\hat a \hat b) \mid \ot(ab)$, so we can get rid of the hats on $a, b$ in the above display, to deduce that
		\begin{equation*}
			\hat c_k^{\;\ot(ab) \cdot C'(\gamma - 1)} = \hat 1\; \text{ for all } \; k\in \{ 2,\dots ,r\} .
		\end{equation*}
		This translates to the fact that the power $\ot(ab) \cdot C'(\gamma - 1)$ takes all $c_k$ ($2 \le k \le r$) into the center $Z(G)$. Again, in order to apply Lemma \ref{lemcenter}, we must also prove this for $c_0 = a$ and $c_1 = b$. Indeed, by Theorem \ref{hallform} and (\ref{caramel}) we see that
		\begin{align*}
			1  &=  (ab)^{\ot(ab) \cdot C(\gamma - 1)} =a^{\ot(ab) \cdot C(\gamma - 1)}b^{\ot(ab) \cdot C(\gamma - 1)}\cdot c_2^{f_2(\ot(ab) \cdot C(\gamma-1))} \cdots c_r^{f_r(\ot(ab) \cdot C(\gamma-1))}\\ &= a^{\ot(ab) \cdot C(\gamma - 1)}b^{\ot(ab) \cdot C(\gamma - 1)}z,
		\end{align*}
		with $z \in Z(G)$. In particular, $b^{\ot(ab) \cdot C(\gamma - 1)} = a^{-\ot(ab) \cdot C(\gamma - 1)}z^{-1}$ commutes with $a$, and similarly, $a^{\ot(ab) \cdot C(\gamma - 1)} = z^{-1}b^{-\ot(ab) \cdot C(\gamma - 1)}$ commutes with $b$. So since $a, b$ generate $G$, we can happily conclude that the power $\ot(ab) \cdot C(\gamma - 1)$ takes all $c_k$ ($0 \le k \le r$) to $Z(G)$. Now the hypotheses of Lemma \ref{lemcenter} are satisfied, so we find that the power $\ot(ab) \cdot C(\gamma - 1) \cdot A(\gamma)$ kills all $c_k$ ($2 \le k \le r$). This means that we can choose
		\begin{equation*}
			C'(\gamma) := C(\gamma - 1) \cdot A(\gamma) = C'(\gamma - 1) \cdot A(\gamma) \cdot (\gamma-1)!
		\end{equation*}
		to complete the inductive step. Unrolling this recurrence, we can now write
		\begin{equation}\label{C(gama)}
			C(\gamma) =  \gamma! \cdot \prod_{i=2}^{\gamma} A(i) \cdot (i-1)!,
		\end{equation}
		which has only prime factors at most $\gamma$. This completes the proof of the theorem.
	\end{proof}

	We will restate here Theorem \ref{constantacontrolabila}, which now is easily seen as an immediate application of Theorem \ref{constanteleBC}.
	
	\begin{theorem}\label{constantacontrolabilaAICI}
		Let $G$ be a finite nilpotent group of class $\gamma $. There exists a finite set of positive rational numbers $S = S(\gamma)$, depending solely on the nilpotency class $\gamma$, and whose elements in reduced form contain no primes exceeding $\gamma$ in their numerators and denominators, such that for all $a, b$ in $G$, there exists $s \in S$ depending on $a$ and $b$ such that
		\begin{equation*}
			\ot(ab) = \ot(a, b) \cdot s.
		\end{equation*}
	\end{theorem}

	\begin{proof}[\quad Proof of Corollary \ref{corolarcugamma}]
		Let $\gamma $ be the nilpotency class of $G$. By Theorem \ref{constantacontrolabilaAICI}, for a pair of elements $a,b$ in $G$ there exists an element $s\in S$, say $s=\frac{\alpha }{\beta }$ with $\gcd(\alpha ,\beta )=1$, such that $\ot(ab) = \ot(a, b)\cdot s$. In particular this implies $\alpha \mid \ot(ab)$ and $\beta \mid \ot(a,b)$. Thus $\alpha $ must be a divisor of $|G|$, and since $\alpha $ has prime factors at most $\gamma $ while all the prime factors of $|G|$ exceed $\gamma $, we deduce that $\alpha $ must in fact be equal to $1$. Now since $\ot(a,b)\mid \lcm(\ot(a),\ot(b))$, it follows that $\ot(a,b)$ must be a divisor of $|G|$. Thus $\beta \mid |G|$, and by the same argument above we conclude that $\beta $ too must be equal to $1$, which completes the proof.
	\end{proof}

	Another consequence of Theorem \ref{constantacontrolabilaAICI} is the following result that gives some information on the order of the commutator of two elements of finite order in nilpotent groups. 

	\begin{corollary}\label{exponentul}
		Let $G$ be a finite nilpotent group of class $\gamma$, containing two elements $a,b$. Let $n$ be the smallest positive exponent such that $a^{n}$ commutes with $b$. Then
		\begin{equation*}
			\ot([a, b]) = n \cdot s
		\end{equation*}
		for some $s \in S(\gamma)$.
	\end{corollary}
	
	\begin{proof} 
		We apply Theorem \ref{constantacontrolabilaAICI} for $x = a^{-1}$ and $y = b^{-1}ab$, to get
		\begin{equation*}
			\ot(a^{-1} b^{-1}ab) = \ot(a^{-1}, b^{-1}ab)  \cdot s.
		\end{equation*}
		Now we recall that $\ot(a^{-1}, b^{-1}ab)$ is precisely the smallest positive exponent $n$ such that $a^n$ commutes with $b$ (i.e. $a^{-n}b^{-1}a^nb = 1$).
	\end{proof}

	We mention here that P. Hall obtained in \cite[4.27]{Hall1} a result related to Corollary \ref{exponentul}, which in particular implies that $a^{\ot([a, b])}$ commutes with $b$, in the case of regular $p$-groups.

	Theorems \ref{constanteleBC} and \ref{JungAGen} have the following immediate consequence, which presents two divisibilities that $\ot(ab)$ satisfies in finite nilpotent groups of class $\gamma$.
	
	\begin{corollary}\label{incadrareNilp}
		Let $G$ be a finite nilpotent group of class $\gamma $, and $a,b$ elements in $G$ of orders $m$ and $n$, respectively, and with $|\langle a\rangle \cap \langle b\rangle |=e$. Let also $D$ and $\varepsilon$ be the same as in Theorem \ref{JungA}. Then $\ot(ab)$ satisfies the divisibilities
		\begin{equation}\label{grosierele}
			\ot(ab)\mid \lcm(m,n)\cdot \frac{B(\gamma )}{\varepsilon }\quad \text{and}\quad \frac{\lcm(m,n)}{D}\mid \ot(ab)\cdot C(\gamma ),
		\end{equation}
		with the integer constants $B(\gamma )$ and $C(\gamma )$ given by (\ref{B(gama)}) and (\ref{C(gama)}), respectively.
	\end{corollary}

	\begin{proof}
		We recall that by Theorem \ref{JungAGen}, the mutual order of our elements $a,b$ satisfies the divisibilities
		\begin{equation}\label{incadrarea}
			\frac{\lcm(m,n)}{D}\mid \ot(a,b) \mid \frac{\lcm(m,n)}{\varepsilon }.
		\end{equation}
		On the other hand, Theorem \ref{constanteleBC} guarantees the existence of the two constants $B(\gamma )$ and $C(\gamma )$ given by (\ref{B(gama)}) and (\ref{C(gama)}), respectively, with $A(\gamma )$ given by Lemma \ref{lemcenter}, and such that
		\begin{equation*}
			\ot(ab) \mid \ot(a, b) \cdot B(\gamma )\quad \text{and}\quad
		\ot(a, b) \mid \ot(ab) \cdot C(\gamma ).
		\end{equation*}
	Using now (\ref{incadrarea}), we immediately deduce that $\ot(ab)$ satisfies the divisibilities (\ref{grosierele}).
	\end{proof}

	\begin{remark}\label{calaJung}
		By (\ref{incadrarea}) we see now that in finite groups of class $\gamma$, sharper estimates for $B(\gamma )$ and $C(\gamma )$ will lead to sharper estimates for the order of the product of two arbitrary elements, as the constants $B$ and $C$ do not depend on the elements $a$ and $b$ that we choose. 
	\end{remark}
	
	\section{Deeper analysis for nilpotent groups of class 2} \label{SClass2}

	A direct application of Theorem \ref{constantacontrolabilaAICI} to the case $\gamma = 2$, using the explicit bounds $A, B, C$ constructed in the previous section gives
	\begin{equation*}
		B(2) = C(2) = 16 \Longrightarrow \frac{\ot(ab)}{\ot(a, b)} \in \{\tfrac 1{16}, \tfrac 18, \ldots, 8, 16\}.
	\end{equation*}
	As we will see, with more careful considerations we can reduce the set to just $\{\frac 12, 1, 2\}$! So already for $\gamma = 2$, our bounds are very weak. One reason for this is that Hall's Theorem \ref{hallform} does not eliminate the degenerate commutators $[a, a], [b, b]$, or the duplicate $[a, b] = [b, a]^{-1}$ form the list. Hence, significant improvements may potentially be achieved by reducing the number of factors appearing in Hall's formula. Also, our construction of $A, B, C$ is quite wasteful, since it assumes the worst at all times (e.g. we work with uniform bounds on everything). We also mention without any proof that for $\gamma = 3$, it seems that the constants $B, C$ could be reduced to 12. We thus believe that there is much potential in studying the asymptotic behavior of the optimal constants $B$ and $C$, although we do not study it in this paper.

	Let us now direct our attention to the case $\gamma = 2$. It is easily seen, from the definition of nilpotency in terms of the lower central series, that a group $G$ has class 2 (or lower) if and only if $G' = [G, G]$ is central. In particular, if $a$ and $b$ are two elements of a nilpotent group $G$ of class 2, then the commutator $[a, b] = a^{-1}b^{-1}ab$ will commute with both $a$ and $b$. This leads us to two very useful results, which are enclosed in the following famous lemma. The first part may be seen as a particularization of Halls' Theorem \ref{hallform}, but in which we already remove the degenerates $[a, a], [b, b]$ and the duplicate $[a, b]$. In order to keep this paper self-contained, we will also include a proof of this lemma.
		
	\begin{lemma}\cite[Lemma 2.2]{Gorenstein}\label{lemacom}
		Let $a, b$ be elements in a group of class 2. Then the following identities hold:	
			\quad (i) $(ab)^n = a^nb^n[b, a]^{n \choose 2}$ for all positive integers $n$,
	
			\quad (ii) $[a^i, b^j] = [a, b]^{ij}$ for all integers $i, j$.
	\end{lemma}
	
	\begin{proof}
		(i) We prove this statement by induction. The base case $n = 1$ is obvious, so let us assume that $(ab)^n = a^nb^n[b, a]^{n \choose 2}$ and prove it for $n + 1$. Indeed,
		\begin{align*}
			(ab)^{n+1} &= aba^nb^n[b, a]^{n \choose 2} = a \cdot ba^nb^{-1} \cdot b^{n + 1}[b, a]^{n \choose 2} \\&= a \cdot (bab^{-1})^n \cdot b^{n + 1}[b, a]^{n \choose 2} =  a \cdot (ab[b, a]b^{-1})^n \cdot b^{n + 1}[b, a]^{n \choose 2} \\ &= a \cdot (a[b, a])^n \cdot b^{n+1}[b, a]^{n \choose 2} = a^{n+1} b^{n+1} [b, a]^{n + 1 \choose 2},
		\end{align*}
		where we have implicitly used the fact that $[b, a]$ is central.
			
		(ii) First consider the case that $j = 1$, where 
		\begin{equation*}
			[a^i, b] = a^{-i}b^{-1}a^{i}b = a^{-i}(b^{-1}ab)^i = a^{-i}(a[a, b])^i = [a, b]^i.
		\end{equation*}
		Now applying this twice, we have 
		\begin{equation*}
			[a^i, b^j] = [a, b^j]^i = [b^j, a]^{-i} = [b, a]^{-ij} = [a, b]^{ij}.
		\end{equation*}
		as desired.
	\end{proof}
	
	Next, we will recall an elementary consequence of Lemma \ref{lemacom}, which allows one to obtain a formula for the order of the commutator of two elements in a finite group of class $2$.
	
	\begin{corollary}\label{ordincomutator}
		Let $G$ be a finite group of class $2$, $a,b$ elements of $G$ of orders $m$ and $n$, respectively, with $|\langle a\rangle\cap\langle b\rangle|=f$, and let $c := [a, b] = a^{-1}b^{-1}ab$. Then
		\begin{equation*}
			\ot(c)=\frac{m}{|\langle a\rangle\cap C_{G}(b)|}=\frac{n}{|\langle b\rangle\cap C_{G}(a)|}
		\end{equation*}
	is the least exponent $i$ such that $a^i$ commutes with $b$, and also the least exponent $j$ such that $b^j$ commutes with $a$. In particular, $\ot(c)\mid \frac{\gcd(m,n)}{f}$, and if any two of $m,n$ and $\ot(c)$ are coprime, then $a$ and $b$ must commute.
	\end{corollary}

	\begin{proof}
		By Lemma \ref{lemacom} (ii) with $j=1$ we have $[a^{i},b]=c^{i}$ for all $i$, so
		\begin{align*}
			\ot(c) &= {\rm min}
			\{i\ge 1:[a^{i},b]=1\}\\
			&= {\rm min}
			\{i\ge 1:a^{i}\in C_{G}(b)\} \\
			&= {\rm min}
			\{i\ge 1:a^{i}\in \langle a\rangle \cap C_{G}(b)\}.
		\end{align*}
		Thus $\ot(c)$ is the order of $\hat{a}$ in the factor group $\langle a\rangle / \langle a\rangle\cap C_{G}(b)$, that is $\frac{m}{|\langle a\rangle\cap C_{G}(b)|}$. Also, since $[a^i, b] = 1$ is essentially equivalent to saying that $a^i$ commutes with $b$, we get that $\ot(c)$ is the least positive exponent $i$ such that $a^i$ commutes with $b$. Using now Lemma \ref{lemacom} (ii) with $i=1$ we see that $[a,b^j]=c^{j}$ for all $j$, and we deduce in a similar way that $\ot(c)$ is also equal to $\frac{n}{|\langle b\rangle\cap C_{G}(a)|}$, and that it is the least positive exponent $j$ such that $b^j$ commutes with $a$.
	\end{proof}

	We will restate here Theorem \ref{IntroducereClass2}, which is our main result in this section.
	
	\begin{theorem}\label{1Class2}
		Let $a, b$ be elements of finite order in a nilpotent group of class 2, and let $r$ be the order of $c = [b, a] := b^{-1}a^{-1}ba$. Then $r$ divides both $\ot(ab)$ and $\ot(a, b)$, and one has the formula
		\begin{equation}\label{turnip}
			\ot(ab) = \ot(a, b) \cdot \frac{\ot(a^rb^rc^{r \choose 2})}{\ot(a^rb^r)},
		\end{equation}
		where the factor $\frac{\ot(a^rb^rc^{r \choose 2})}{\ot(a^rb^r)}$ lies in the set $\{\frac 12, 1, 2\}$.
	\end{theorem}
	
	\begin{proof}
		To prove that $r \mid \ot(ab)$, note that the equality $1 = (ab)^{\ot(ab)}$ can be further written via Lemma \ref{lemacom} (i) as
		\begin{equation*}
			1 = a^{\ot(ab)}b^{\ot(ab)}[b, a]^{\ot(ab) \choose 2},
		\end{equation*}
		which in particular implies that $a^{\ot(ab)}$ commutes with $b$. But then by Lemma \ref{lemacom} (ii), we observe that $[b, a]^{\ot(ab)} = [b, a^{\ot(ab)}] = 1$, so indeed $r \mid \ot(ab)$. To prove that $r \mid \ot(a, b)$ we reason similarly: The equation $1 = a^{\ot(a, b)}b^{\ot(a, b)}$ implies that $a^{\ot(a, b)}$ commutes with $b$, which in turn implies via Lemma \ref{lemacom} (ii) that $[b, a]^{\ot(a, b)} =[b, a^{\ot(a,b)}]= 1$, as needed. 
			
		Equation (\ref{turnip}) follows now after a straightforward computation, by repeatedly applying Corollary \ref{reduce} and observing that $\ot(a^r, b^r) = \ot(a^rb^r)$ (as $a^{r}$ commutes with $b^{r}$):
		\begin{equation*}
			\ot(ab) = r\cdot \ot((ab)^r)=r \cdot \ot(a^rb^rc^{r \choose 2}) = r \cdot \ot(a^r, b^r) \cdot \frac{\ot(a^rb^rc^{r \choose 2})}{\ot(a^r, b^r)} = \ot(a, b) \cdot  \frac{\ot(a^rb^rc^{r \choose 2})}{\ot(a^rb^r)}.
		\end{equation*}
	
		Lastly, we show that the factor $\frac{\ot(a^rb^rc^{r \choose 2})}{\ot(a^rb^r)}$ lies in $\{\frac 12, 1, 2\}$. Let us denote $x := a^rb^r$, $y := c^{r \choose 2}$, and notice that $y$ has order 1 or 2. If $y$ has order 1, then the fraction trivially equals 1. Thus, we may assume that $\ot(y) = 2$.  If we denote $\ot(x)$ by $t$, then we notice that Jungnickel's Theorem \ref{JungA} gives
		\begin{equation*}
			\frac{\lcm(t, 2)}{D} \mid \ot(a^rb^rc^{r \choose 2}) \mid \frac{\lcm(t, 2)}{\varepsilon},
		\end{equation*}
		where $D, \varepsilon$ are associated to the pair $(x, y)$ (and not $(a, b)$). When $t$ is odd, we have $D = \varepsilon = 1$ and $\lcm(t, 2) = 2t$, so $\frac{\ot(a^rb^rc^{r \choose 2})}{\ot(a^rb^r)} = 2$. When $t$ is even, we have $D = \varepsilon \in \{1, 2\}$, and $\lcm(t, 2) = t$. So in this case $\frac{\ot(a^rb^rc^{r \choose 2})}{\ot(a^rb^r)} = 2$ lies in the set $\{\frac 12, 1\}$. This finishes the proof.
	\end{proof}
	
	In particular, we obtain three corollaries, corresponding to the three possible values of $\ot(ab)/\ot(a,b)$.
	
	\begin{corollary}\label{CoroCl2.1}
		Let $a, b$ be elements of finite order in a nilpotent group of class 2, and let $r$ be the order of $c = [b, a] := b^{-1}a^{-1}ba$. If $r$ is odd, then
		\begin{equation*}
			\ot(ab) = \ot(a, b),
		\end{equation*}
	just like in the abelian case.
	\end{corollary}

	\begin{proof}
		Following the analysis in from the main theorem, this is the case in which $y$ has order 1, and the ratio $\frac{\ot(a^rb^rc^{r \choose 2})}{\ot(a^rb^r)}$ is invariably equal to 1.
	\end{proof}

	\begin{corollary}\label{CoroCl2.2}
		Let $a, b$ be elements of finite order in a nilpotent group of class 2, and let $r$ be the order of $c = [b, a] := b^{-1}a^{-1}ba$. If $r$ is even, and $\frac 1r \cdot \ot(a, b)$ is odd, then
		\begin{equation*}
			\ot(ab) = 2 \cdot \ot(a, b).
		\end{equation*}
	\end{corollary}

	\begin{proof}
		This corresponds to the case that $y$ has order 2, but $t$ is odd. In this particular case, we got $\ot(ab) = 2 \cdot \ot(a, b)$.
	\end{proof}
	
	\begin{corollary}\label{CoroCl2.3}
		Let $a, b$ be elements of finite order in a nilpotent group of class 2, and let $r$ be the order of $c = [b, a] := b^{-1}a^{-1}ba$. If $r$ is even, and $\frac 1r \cdot \ot(a, b)$ is even, then
		\begin{equation}
			\ot(ab) = \begin{cases}
				\frac 12 \cdot \ot(a, b)	& \text{if }\; a^{\frac 1{2}\ot(a, b)}b^{\frac 1{2}\ot(a, b)} = c^{\frac r2} \text{ and } \frac 1{2r}\ot(a, b) \text{ is odd,}\\
				\ot(a, b)					& \text{if }\; a^{\frac 1{2}\ot(a, b)}b^{\frac 1{2}\ot(a, b)} \neq c^{\frac r2} \text{ or } \frac 1{2r}\ot(a, b) \text{ is even.}
			\end{cases}
		\end{equation}
	\end{corollary}
	
	\begin{proof} 
		Finally, this is the case that $y$ has order 2, and $t$ is even. Recall that here the value of $D$ dictates the value of the ratio $\frac{\ot(a^rb^rc^{r \choose 2})}{\ot(a^rb^r)}$. That is, $\ot(ab) = \ot(a, b)$ when $D = 1$, and $\ot(ab) = \frac 12 \cdot \ot(a, b)$ when $D = 2$. To discern between these two cases, we recall the definition of $D$: it is the largest divisor of $|\langle x \rangle \cap \langle y \rangle|$, that is coprime to both $\frac{t}{\gcd(t, 2)} = \frac t2$ and $\frac{2}{\gcd(t, 2)} = 1$.
		
		If $\frac 1{2r} \ot(a, b) = \frac t2$ is even, then $D$ will be required to be odd. On the other hand, $|\langle x \rangle \cap \langle y \rangle|$ must divide $\ot(y) = 2$, so $D = 1$ and $\ot(ab) = \ot(a, b)$ as promised. Next, if $a^{\frac 1{2}\ot(a, b)}b^{\frac 1{2}\ot(a, b)} \neq c^{\frac r2}$, then essentially $\langle x \rangle \cap \langle y \rangle$ is trivial, and once again $D = 1$, i.e. $\ot(ab) = \ot(a, b)$.
		
		The remaining case is when both $\frac t2$ is odd, and $a^{\frac 1{2}\ot(a, b)}b^{\frac 1{2}\ot(a, b)} = c^{\frac r2}$. It is easily seen that in this case $|\langle x \rangle \cap \langle y \rangle| = 2$, and since $\frac t2$ is odd we get $D = 2$, i.e. $\ot(ab) = \frac 12 \cdot \ot(a, b)$ as wished.
	\end{proof}
	
	\begin{remark}
		A natural question is whether all three values in $\{\frac 12, 1, 2\}$ can be realized as the ratio $\ot(ab)/\ot(a, b)$, with $a, b$ elements of a finite nilpotent group of class two. The value $s = 1$ obviously appears, for example when $a$ and $b$ commute. The other two show up in the diherdal group $D_4 = \langle r, s\;|\; r^4, s^2, rsrs\rangle$:
		\begin{gather*}
			\ot(rs, s) = 2, \quad \ot(rs \cdot s) = 4 \\
			\ot(r, s) = 4, \quad \ot(r \cdot s) = 2.
		\end{gather*}
		We leave the verification of these identities to the reader.
	\end{remark}

	\bibliographystyle{acm}
	\bibliography{OrdNilp}

\begin{thebibliography}{10}

\bibitem{Adian}
{\sc Adian, S.~I.}
\newblock {\em The {B}urnside problem and identities in groups}, vol.~95 of
  {\em Ergebnisse der Mathematik und ihrer Grenzgebiete [Results in Mathematics
  and Related Areas]}.
\newblock Springer-Verlag, Berlin-New York, 1979.
\newblock Translated from the Russian by John Lennox and James Wiegold.

\bibitem{Baer}
{\sc Baer, R.}
\newblock Nilpotent groups and their generalizations.
\newblock {\em Trans. Amer. Math. Soc. 47\/} (1940), 393--434.

\bibitem{Bavard}
{\sc Bavard, C.}
\newblock Longueur stable des commutateurs.
\newblock {\em Enseign. Math. (2) 37}, 1-2 (1991), 109--150.

\bibitem{Calegari}
{\sc Calegari, D.}
\newblock {\em scl}, vol.~20 of {\em MSJ Memoirs}.
\newblock Mathematical Society of Japan, Tokyo, 2009.

\bibitem{Clement}
{\sc Clement, A.~E., Majewicz, S., and Zyman, M.}
\newblock {\em The theory of nilpotent groups}.
\newblock Birkh\"{a}user/Springer, Cham, 2017.

\bibitem{Culler}
{\sc Culler, M.}
\newblock Using surfaces to solve equations in free groups.
\newblock {\em Topology 20}, 2 (1981), 133--145.

\bibitem{Golod}
{\sc Golod, E.~S.}
\newblock On nil-algebras and finitely approximable {$p$}-groups.
\newblock {\em Izv. Akad. Nauk SSSR Ser. Mat. 28\/} (1964), 273--276.

\bibitem{GS}
{\sc Golod, E.~S., and \v{S}afarevi\v{c}, I.~R.}
\newblock On the class field tower.
\newblock {\em Izv. Akad. Nauk SSSR Ser. Mat. 28\/} (1964), 261--272.

\bibitem{Gorenstein}
{\sc Gorenstein, D.}
\newblock {\em Finite groups}, second~ed.
\newblock Chelsea Publishing Co., New York, 1980.

\bibitem{Hall1}
{\sc Hall, P.}
\newblock A {C}ontribution to the {T}heory of {G}roups of {P}rime-{P}ower
  {O}rder.
\newblock {\em Proc. London Math. Soc. (2) 36\/} (1934), 29--95.

\bibitem{Magidin}
{\sc (https://math.stackexchange.com/users/742/arturo magidin), A.~M.}
\newblock Why do the elements of finite order in a nilpotent group form a
  subgroup?
\newblock Mathematics Stack Exchange.
\newblock URL:https://math.stackexchange.com/q/79687 (version: 2019-06-21).

\bibitem{Ivanov}
{\sc Ivanov, S.~V.}
\newblock The free {B}urnside groups of sufficiently large exponents.
\newblock {\em Internat. J. Algebra Comput. 4}, 1-2 (1994), ii+308.

\bibitem{IvOls}
{\sc Ivanov, S.~V., and Ol'shanski\u{\i}, A.~Y.}
\newblock Hyperbolic groups and their quotients of bounded exponents.
\newblock {\em Trans. Amer. Math. Soc. 348}, 6 (1996), 2091--2138.

\bibitem{Jacobson}
{\sc Jacobson, N.}
\newblock {\em Basic algebra. {I}}, second~ed.
\newblock W. H. Freeman and Company, New York, 1985.

\bibitem{Jungnickel3}
{\sc Jungnickel, D.}
\newblock {\em Finite fields}.
\newblock Bibliographisches Institut, Mannheim, 1993.
\newblock Structure and arithmetics.

\bibitem{Jungnickel1}
{\sc Jungnickel, D.}
\newblock On the order of a product in a finite abelian group.
\newblock {\em Math. Mag. 69}, 1 (1996), 53--57.

\bibitem{Mazurov}
{\sc Khukhro, E.~I., and Mazurov, V.~D.}
\newblock Unsolved problems in group theory. the kourovka notebook, 2014.

\bibitem{Konig}
{\sc K\"{o}nig, J.}
\newblock A note on the product of two permutations of prescribed orders.
\newblock {\em European J. Combin. 57\/} (2016), 50--56.

\bibitem{Kostrikin}
{\sc Kostrikin, A.~I.}
\newblock {\em Around {B}urnside}, vol.~20 of {\em Ergebnisse der Mathematik
  und ihrer Grenzgebiete (3) [Results in Mathematics and Related Areas (3)]}.
\newblock Springer-Verlag, Berlin, 1990.
\newblock Translated from the Russian and with a preface by James Wiegold.

\bibitem{Luneburg}
{\sc L\"{u}neburg, H.}
\newblock {\em On the rational normal form of endomorphisms}.
\newblock Bibliographisches Institut, Mannheim, 1987.
\newblock A primer to constructive algebra.

\bibitem{Lysenok}
{\sc Lys\"{e}nok, I.~G.}
\newblock Infinite {B}urnside groups of even period.
\newblock {\em Izv. Ross. Akad. Nauk Ser. Mat. 60}, 3 (1996), 3--224.

\bibitem{Milne}
{\sc Milne, J.~S.}
\newblock Group theory (v3.13), 2013.
\newblock Available at www.jmilne.org/math/.

\bibitem{Ols}
{\sc Ol'shanski\u{\i}, A.~Y.}
\newblock {\em Geometry of defining relations in groups}, vol.~70 of {\em
  Mathematics and its Applications (Soviet Series)}.
\newblock Kluwer Academic Publishers Group, Dordrecht, 1991.
\newblock Translated from the 1989 Russian original by Yu. A. Bakhturin.

\bibitem{Pan}
{\sc Pan, J.}
\newblock On a conjecture about orders of products of elements in the symmetric
  group.
\newblock {\em J. Pure Appl. Algebra 222}, 2 (2018), 291--296.

\bibitem{Waerden}
{\sc van~der Waerden, B.~L.}
\newblock {\em Algebra. {V}ol. {I}}.
\newblock Springer-Verlag, New York, 1991.
\newblock Based in part on lectures by E. Artin and E. Noether, Translated from
  the seventh German edition by Fred Blum and John R. Schulenberger.

\bibitem{Zelmanov1}
{\sc Zelmanov, E.~I.}
\newblock Solution of the restricted {B}urnside problem for groups of odd
  exponent.
\newblock {\em Izv. Akad. Nauk SSSR Ser. Mat. 54}, 1 (1990), 42--59, 221.

\bibitem{Zelmanov2}
{\sc Zelmanov, E.~I.}
\newblock Solution of the restricted {B}urnside problem for {$2$}-groups.
\newblock {\em Mat. Sb. 182}, 4 (1991), 568--592.

\end{thebibliography}
	
\end{document}